\documentclass[11pt, reqno, a4paper]{amsart} 
\usepackage[english]{babel}
\usepackage{amsfonts, amsmath, amsthm, amssymb,amscd,indentfirst}
\usepackage{times}
\usepackage{color}
\usepackage{enumerate}
\usepackage{anysize} %Anysize is a simple package that can be used to set up the margins of a document.
\usepackage{mathrsfs}
\marginsize{1.8cm}{1.8cm}{1.5cm}{1.5cm}
\newtheorem{theorem}{Theorem}[section]
\newtheorem{proposition}{Proposition}[section]
\newtheorem{lemma}{Lemma}[section]
\newtheorem{definition}{Definition}[section]

\newtheorem{ass}{Assumption}

\newcommand{\einf}{\operatorname{ess} \operatorname{inf}}

\newcommand{\dx}{\textnormal{d}x}
\newcommand{\ds}{\textnormal{d}\sigma}

\newcommand{\spt}{\operatorname{supp}}

\newcommand{\ma}{\mathcal{M}}
\newcommand{\wma}{\mathcal{W}}

\title{Some eigenvalue problems involving the $(p(\cdot),q(\cdot))$-Laplacian}
\author{Juan Alcon Apaza}
\address{Universidade Federal Fluminense, Instituto de Matemática, Campus do Gragoatá, Rua Prof. Marcos Waldemar de Freitas, s/n, bloco H, Niterói, RJ 24210-201, Brazil }
\email{jpablo@id.uff.br}

\begin{document}
\maketitle

\begin{abstract}
In this work, we are concerned with a Robin and Neumann problem with $(p(\cdot), q(\cdot))$-Laplacian. Under some appropriate conditions on the data involved in the elliptic problem, we prove the existence of  solutions applying two versions of Mountain Pass theorem, Ekeland's variational principle and Lagrange multiplier rule. 
\end{abstract}

\let\thefootnote\relax
\let\thefootnote\relax
\section{introduction}

The aim of this paper is to investigate the eigenvalue problem

\begin{equation}\label{1}
\left\{
\begin{aligned}
-\Delta _p u - \Delta_q u &= \lambda \alpha \vert    u\vert    ^{r-2} & &  \text { in } \Omega,\\
 \left( \vert    \nabla u\vert    ^{p-2}  + \vert    \nabla u \vert    ^{q-2}\right)\frac{\partial u}{\partial \nu} + \beta _1 \vert    u\vert    ^{p-2} + \beta _2 \vert    u\vert    ^{q-2}&=0 & & \text { on } \partial \Omega,
\end{aligned}
\right.
\end{equation}
where $\Omega \subset \mathbb{R} ^n$, $n\geqslant 2$, is a bounded smooth domain, and $\frac{\partial u}{\partial \nu}$ is the outer unit normal derivative on $\partial \Omega$. We study \eqref{1} in  distinct situations, and the solutions $u$ will be sought in the variable exponent Sobolev space $\wma := W^{1,\ma (\cdot)} (\Omega)$, where $\ma = \max \{p,q\}$.

The main results of this work are the following theorems.

\begin{theorem} \label{52} Assume that $p, q, r \in C_{+}(\bar{\Omega})$, $\alpha \in L^\infty (\Omega)$, $\alpha ^- >0$, and $\beta _1, \beta _2 \in L^\infty (\partial \Omega)$ with $\beta _1 ^-, \beta ^- _2 >0$.
\begin{enumerate}[(A)]
\item \label{3} If $r^+ < \min\{p^- , q^-\}$, then any $\lambda>0$ is an eigenvalue for problem \eqref{1}. Moreover, for any $\lambda>0$ there exists a sequence $\{u_k \}$ of nontrivial weak solutions for problem \eqref{1} such that $u_k \rightarrow 0$ in $\wma$.

\item \label{4} If  $r^- <  \min \{p^- , q^-\}$  and   $r^+ <\left( \ma ^\ast \right)^-$, then there exists $\Lambda >0$ such that any $\lambda \in\left(   0 , \Lambda \right)$ is an eigenvalue for problem \eqref{1}.

\item \label{5} If $\ma^{+}< r^- \leqslant r^+ <(\ma ^\ast) ^-$, then for any $\lambda>0$, the problem \eqref{1} possesses a nontrivial weak solution.
\end{enumerate}

\end{theorem}

\begin{theorem}\label{53}Assume that $p\in C_+ (\bar{\Omega})$, $r\equiv q$, $q>2$ is a constant, $\alpha \in L^\infty (\Omega)$, $\alpha ^- >0$, and  $\beta _1\equiv \beta _2\equiv  0 $.

\begin{enumerate}[(A)]
\item \label{28}  Suppose  that $p^+<q$. Then the eigenvalue set of problem \eqref{1} is precisely $\{0\} \cup \left(\inf _{u\in \mathcal{C}_q \backslash \{0\}}\frac{\int _\Omega \vert \nabla u\vert ^q \dx}{\int _\Omega \alpha\vert  u\vert ^q \dx}  , \infty\right)$.

\item \label{43} Suppose that   $q<p^-$. Then the eigenvalue set of problem \eqref{1} is precisely $\{0\} \cup \left(\inf _{u\in \mathcal{C} \backslash \{0\}}\frac{\int _\Omega \vert \nabla u\vert ^q \dx}{\int _\Omega \alpha\vert  u\vert ^q \dx}  , \infty\right)$.

\end{enumerate}

\end{theorem}

\section{Preliminaries}

We first recall some facts on the variable exponent spaces $L^{p(\cdot)} (\Omega)$ and $W^{1,p(\cdot)} (\Omega)$. For more detail, see \cite{fan2,  fanxi, kova}. Suppose that $\Omega\subset \mathbb{R} ^n$ is a bounded open domain with smooth boundary $\partial \Omega$ and $p\in C_+ (\bar{\Omega})$, where
$$
C_+ (\bar{\Omega}) = \left\{ p\in C(\bar{\Omega}) \: \vert \: \inf _{\bar{\Omega}} p > 1\right\}.
$$
The variable exponent Lebesgue space $L^{p(\cdot)} (\Omega)$ is defined  by
$$
L^{p(\cdot)} (\Omega) = \left\{ u: \Omega \rightarrow \mathbb{R} \: \vert  \:  u \text { is measurable  and }  \int _{\Omega} \vert u\vert ^{p} \dx <  \infty  \right\},
$$ 
with the norm
$$
\Vert u\Vert _{L^{p(\cdot)} (\Omega)} = \inf \left\{ \tau > 0 \: \vert  \:  \int _{\Omega} \left\vert \frac{u}{\tau}\right\vert ^p \dx \leqslant 1 \right\}.
$$
The variable exponent Sobolev space $W^{1,p(\cdot)} (\Omega)$ is defined by
$$
W^{1,p(\cdot)} (\Omega)= \left\{u\in L^{p(\cdot)} (\Omega) \: \vert  \: \vert \nabla u \vert \in L^{p(\cdot)} (\Omega)\right\},
$$ 
with the norm
$$
\Vert  u\Vert _{W^{1,p(\cdot)} (\Omega)}  =\Vert    \nabla u\Vert  _{L^{p(\cdot)}  (\Omega)} + \Vert u\Vert    _{L^{p(\cdot)} (\Omega)} .
$$
Both $L^{p(\cdot)}(\Omega)$ and $W^{1, p(\cdot)}(\Omega)$ are separable, reflexive and uniformly convex Banach spaces, see \cite{fanxi, kova}.

\begin{proposition}\label{31} (see \cite[Theorem 2.1]{kova}).
For any $u \in L^{p(\cdot)}(\Omega)$ and $v \in L^{\frac{p}{p-1}(\cdot)}(\Omega)$,  we have
$$
\int_{\Omega}\vert    u v\vert     \dx \leq 2\Vert  u\Vert  _{L^{p(\cdot)} (\Omega)}\Vert  v\Vert  _{L^{\frac{p}{p-1}(\cdot)} (\Omega)}.
$$
\end{proposition}

\begin{proposition} \label{6}(see \cite{fan2, fanxi}).
 If  $q\in C(\bar{\Omega})$ and
$$
1 \leqslant  q < p^\ast \quad \text { on } \bar{\Omega},
$$ 
then there is a compact embedding $W^{1, p(\cdot)}(\Omega) \hookrightarrow L^{q(\cdot)}(\Omega)$, where
$$
p^\ast (x) = \left\{
\begin{aligned}
&\frac{np(x)}{n-p(x)} & & \text { if } p(x)<n,\\
&\infty & & \text { if }  p(x)\geqslant n.
\end{aligned}
\right.
$$
\end{proposition}

Now, we introduce a norm which will be used later. Let $\beta \in L^{\infty}(\partial \Omega)$, with $\beta^{-}:=\einf _{ \partial \Omega} \beta>0$,  and  $u \in W^{1, p(\cdot)}(\Omega)$, define
$$
\Vert  u\Vert  _{p, \beta}=\inf \left\{\tau>0 \:\vert    \: \int_{\Omega} \left\vert    \frac{\nabla u}{\tau}\right\vert    ^{p} \dx+\int_{\partial \Omega} \beta\left\vert    \frac{u}{\tau}\right\vert    ^{p} \ds \leqslant 1\right\},
$$
where $\ds$ is the measure on the boundary $\partial \Omega$. By \cite[Theorem 2.1]{den}, $\Vert  \cdot\Vert  _{p , \beta}$ is also a norm on $W^{1, p(\cdot)}(\Omega)$ which is equivalent to $\Vert  \cdot\Vert  _{W^{1, p(\cdot)}(\Omega)}$.

For $p\in C_+ (\bar{\Omega})$, we will write $p^- := \inf _{\bar{\Omega}} p$ and $p^+ := \sup _{\bar{\Omega}} p$. An important role in manipulating the generalized Lebesgue-Sobolev spaces is played by the mapping defined by the following

\begin{proposition}(see \cite[Proposition 2.4]{den}). Let $\rho _{p , \beta }(u):=\int_{\Omega}\vert    \nabla u\vert    ^{p} \dx +\int_{\partial \Omega} \beta \vert    u\vert    ^{p} \ds $. For $u \in W^{1, p(\cdot)}(\Omega)$ we have
\begin{enumerate}[(i)]
\item $\Vert  u\Vert  _{p ,\beta }<1(=1,>1) \Leftrightarrow \rho _{p , \beta} (u)<1(=1,>1)$.
\item $\Vert  u\Vert  _{p , \beta} \leqslant 1 \Rightarrow\Vert  u\Vert  _{p , \beta}^{p^{+}} \leqslant \rho _{p , \beta}  (u) \leqslant\Vert  u\Vert  _{p , \beta}^{p^{-}}$.
\item $\Vert  u\Vert  _{p , \beta} \geqslant 1 \Rightarrow\Vert  u\Vert  _{p,\beta}^{p^{-}} \leqslant \rho _{p , \beta } (u) \leqslant\Vert  u\Vert  _{p , \beta}^{p^{+}}$.
\end{enumerate}
\end{proposition}

The Euler-Lagrange functional associated with \eqref{1} is  defined as $\Phi_{\lambda}: \wma \rightarrow \mathbb{R}$,
$$
\Phi_{\lambda}(u)=\int_{\Omega} \frac{1}{p}\vert    \nabla u\vert    ^{p} +  \frac{1}{q}\vert    \nabla u\vert    ^{q} \dx +\int_{\partial \Omega} \frac{\beta _1}{p}\vert    u\vert    ^{p} + \frac{\beta _2}{q}\vert    u\vert    ^{q} \ds -\lambda \int_{\Omega} \frac{\alpha}{r}\vert    u\vert    ^r \dx ,
$$
where  $p, q, r \in C_{+}(\bar{\Omega})$, $r^+ < (\ma ^\ast)^- $, $\alpha \in L^\infty (\Omega)$, $\alpha ^- >0$, and $\beta _1, \beta _2 \in L^{\infty} (\partial \Omega)$, with $\beta _1^-, \beta _2 ^->0$ or $\beta _1 \equiv \beta _2 \equiv 0$. Standard arguments imply that $\Phi_{\lambda} \in C^{1}(\wma, \mathbb{R})$ and
\begin{equation} \label{51}
\begin{split}
&\left\langle\Phi_{\lambda}^{\prime}(u), v\right\rangle=\int_{\Omega}\left( \vert    \nabla u\vert    ^{p-2} + \vert    \nabla u\vert    ^{q-2} \right) \nabla u  \nabla v  \dx + \int_{\partial \Omega} \left( \beta _1 \vert    u\vert    ^{p-2} + \beta _2 \vert    u\vert    ^{q-2} \right) u v \ds\\
&-\lambda \int_{\Omega}\alpha \vert    u\vert    ^{r-2} u v \dx,
\end{split}
\end{equation}
for all $u, v \in \wma$. Thus, the weak solutions of \eqref{1} coincide with the critical points of $\Phi_{\lambda}$. If such a weak solution exists and is nontrivial, then the corresponding $\lambda$ is an eigenvalue of problem \eqref{1}.

Finally, we define $L_{p,\beta} : W^{1,p(\cdot)} (\Omega) \rightarrow \left( W^{1,p(\cdot)}(\Omega) \right)^\ast$ by
$$
\left\langle L_{p,\beta}(u),v \right\rangle = \int _{\Omega} \vert \nabla u\vert ^{p-2} \nabla u \nabla v \dx + \int _{\partial \Omega} \beta \vert u\vert ^{p-2} uv \ds,
$$
where $\beta \in L^\infty (\partial \Omega)$ and $\beta ^- >0$.

\begin{proposition}\label{2} (see  \cite[Proposition 2.2]{ge}). 
$L_{p,\beta}$ is a mapping of type ($S_+$), \textit{i.e.}, if $u_{k} \rightharpoonup u$ in  $W^{1,p(\cdot)}(\Omega) $, and $\lim \sup _{k \rightarrow \infty }\left\langle L_{p,\beta} \left(u_k\right) - L_{p,\beta}(u),u_k-u\right\rangle \leqslant 0$. Then  $u_k \rightarrow u$ in $W^{1,p(\cdot)}(\Omega) $.
\end{proposition}

%\begin{rem}
%Noting that $\Phi_{\lambda}^{\prime}$ is still of type ($S_{+}$). Hence, any bounded (PS) sequence of $\Phi_{\lambda}$ in the reflexive Banach space $\wma$ has a convergent subsequence.
%\end{rem}
\section{Main results}

We organize  our main results into five sections. In Case \ref{35}, we prove Theorem \ref{52} \eqref{3}. In Case \ref{36}, we prove Theorem \ref{52} \eqref{4}. In Case \ref{37}, we prove Theorem \ref{52} \eqref{5}. In Case  \ref{38}, we prove Theorem  \ref{53} \eqref{28}. In Case  \ref{44}, we prove Theorem  \ref{53} \eqref{43}. In the first three sections we will follow \cite{allaoui21}, and in the last two sections we follow  \cite{mihua}.

\subsection{The case $r^+ < \min\{p^- , q^-\}$} \label{35} 

We want to apply the symmetric mountain pass lemma in \cite{kajikiya}. We start with the following 

\begin{definition} Let $X$ be a Banach space and $E$ a subset of $X$. $E$ is said to be symmetric if $u \in E$ implies $-u \in E$. For a closed symmetric set $E$ which does not contain the origin, we define a genus $\gamma(E)$ of $E$ by the smallest integer $k$ such that there exists an odd continuous mapping from $E$ to $\mathbb{R}^{k} \backslash\{0\}$. If there does not exist such a $k$, we define $\gamma(E)=\infty$. Moreover, we set $\gamma(\emptyset)=0$. Let $\Gamma_{k}$ denote the family of closed symmetric subsets $E$ of $X$ such that $0 \notin E$ and $\gamma(E) \geqslant k$.
\end{definition}
\begin{ass}
Let $X$ be an infinite dimensional Banach space and $I \in C^{1}(X , \mathbb{R})$ satisfy (A1) and (A2) below.
\begin{enumerate}
\item[(A1)] $I(u)$ is even, bounded from below, $I(0)=0$ and $I(u)$ satisfies the Palais-Smale condition (PS).

\item[(PS)] Any sequence $\left\{u_{k}\right\}$ in $X$ such that $\left\{I\left(u_{k}\right)\right\}$ is bounded and $I^{\prime}\left(u_{k}\right) \rightarrow 0$ in $X^\ast$ as $k \rightarrow \infty$ has a convergent subsequence.

\item[(A2)] For each $k \in \mathbb{N}$, there exists an $E_{k} \in \Gamma_{k}$ such that $\sup _{u \in E_{k}} I(u)<0$.
\end{enumerate}
\end{ass}

\begin{theorem} \label{9}
(Symmetric mountain pass lemma) Under Assumption (A), either (i) or (ii) below holds.
\begin{enumerate}[(i)]
\item There exists a sequence $\{u_k\}$ such that $I^\prime (u_k) = 0$, $I\left(u_{k}\right)<0$ and $\{u_k\}$ converges to zero.  

\item There exist two sequences $\{u_k\}$ and $\{v_k\}$ such that $I^\prime (u_k) = 0$, $I(u_k) =0$, $u_{k} \neq 0 $,  $\lim _{k\rightarrow \infty} u_{k}=0$, $I^\prime (v_k) =0$, $I(v_k) <0$, $\lim _{k\rightarrow \infty} I(v_{k})=0$, and $\{v_k\}$ converges to a non-zero limit.
\end{enumerate}

\end{theorem}

We will need the following result.

\begin{lemma}  \label{32}
\begin{enumerate}[(i)]
\item \label{7} The functional $\Phi_{\lambda}$ satisfies the condition (A1).
\item  \label{8} The functional $\Phi_{\lambda}$ satisfies the condition (A2).
\end{enumerate}
\end{lemma}

\begin{proof}
\eqref{7} It is clear that $\Phi_{\lambda}$ is even and $\Phi_{\lambda}(0)=0$. Let $u\in \wma$. Since $r^{+}< \ma ^{-}$, by Young's inequality, we have
$$
\vert    u\vert    ^r \leqslant \varepsilon \vert    u\vert    ^{\ma} + C_1 \left(\varepsilon , r , \ma\right) \quad \text { on } \Omega.
$$ 
Then
$$
\begin{aligned}
\Phi_{\lambda}(u) & \geqslant C_2 \left(\int_{\Omega}\vert    \nabla u\vert    ^{\ma} \dx+\int_{\partial \Omega} \vert    u\vert    ^{\ma} \ds\right)- \frac{\lambda \alpha ^+}{r^-} \left(\varepsilon  \int _{ \Omega} \vert    u\vert    ^{\ma} \dx + C_1 \vert    \Omega\vert     \right),
\end{aligned}
$$
where $C_2 = \frac{\min \{1,\beta _1 ^- , \beta _2 ^-\}}{\max \{p^+ , q^+\}}$. Choosing $\varepsilon = \frac{r^- C_2}{2\lambda \alpha ^+}$, we have
$$
\Phi_{\lambda}(u)  \geqslant \frac{C_2}{2}  \left(\int_{\Omega}\vert    \nabla u\vert    ^{\ma} \dx+\int_{\partial \Omega} \vert    u\vert    ^{\ma} \ds\right)- \frac{\lambda \alpha ^+ C_1}{r^-}\vert     \Omega\vert    .
$$
Therefore, $\Phi_{\lambda}$ is bounded from below and coercive. It remains to show that the functional $\Phi_{\lambda}$ satisfies the (PS) condition to complete the proof. Let $\{ u_k \} \subset \wma$ be a sequence such that
$$
\{ \Phi_{\lambda}\left(u_k \right) \} \text { is bounded } \quad \text { and } \quad \Phi_{\lambda}^{\prime}\left(u_k \right) \rightarrow 0 \text { in } \wma ^\ast . 
$$
Then, by the coercivity of $\Phi_{\lambda}$, the sequence $\{u_k \}$ is bounded in $\wma$. By the reflexivity of $\wma$ and Proposition \ref{6}, for a subsequence still denoted $\{ u_k \}$, we have
$$
u_k  \rightharpoonup u \text { in } \wma \quad \text { and } \quad u_k \rightarrow u \text { in } L^{r(\cdot)}(\Omega) .
$$
Therefore,
$$
\left\langle\Phi_{\lambda}^{\prime} (u_k) - \Phi_{\lambda}^{\prime} (u ) , u_k-u\right\rangle \rightarrow 0 \quad \text { and } \quad \int_{ \Omega} \alpha \left( \vert   u_k \vert    ^{r-2}u_k - \vert   u \vert    ^{r-2}u \right)\left(u_k-u\right) \dx \rightarrow 0 .
$$
Thus
\begin{equation*}
\left\langle \left(L_{p,\beta _1} + L_{q,\beta _2} \right) (u_k ) - \left(L_{p,\beta _1} + L_{q,\beta _2} \right)(u),u_k-u \right\rangle \rightarrow 0 .
\end{equation*}
On the other hand,  we have  $\left( \vert a \vert ^{\sigma -2} a - \vert b \vert ^{\sigma -2}b\right)\cdot (a-b)\geqslant 0$ for all  $a,b\in \mathbb{R}^{m}\backslash \{0\}$ with $m\geqslant 1$ and  $\sigma >1$. Consequently
\begin{gather*}
u_k \rightharpoonup u \text { in } W^{1,p(\cdot)} (\Omega), \quad  \limsup _{k \rightarrow \infty }\left\langle L_{p,\beta _1} \left(u_k\right) - L_{p,\beta _1}(u),u_k-u\right\rangle \leqslant 0 ,\\
u_k \rightharpoonup u \text { in } W^{1,q(\cdot)} (\Omega), \quad  \text { and } \quad \limsup _{k \rightarrow \infty }\left\langle L_{q,\beta _2} \left(u_k\right) - L_{q,\beta _2}(u),u_k-u\right\rangle \leqslant 0.
\end{gather*}
Therefore,  according to Lemma \ref{2},  $u_k \rightarrow u$ in $\wma$. The proof is complete.\\

%%%%%%%%%%%%%%%%%%%%%%%%%%%%%%%%%%%%%%%%%%%%%%%%%%%%%%%%%%%%%%%%%%%%%%%%%%%%%%%%%%%%%%%%%%%%%%%%%%%%%%%%%%%%%%%%%%%%%%%%%%%%%%%%%%%%%%%%%%%%%%%%%%%%%%%%%%%%%%%%%%%%%%%%%%%%%%%%%%%%%%%%%%%%%%%%%%%%%%%%%%%%%%%%%%

\eqref{8} Let $\{\varphi _k\} _{k=1}^\infty \subset C^\infty _0 (\mathbb{R} ^n)$   be such that  $\spt (\varphi _k)\subset \Omega$, $\{\varphi _k \neq 0\} \neq \emptyset$, and $\spt (\varphi _j) \cap \spt (\varphi _k) = \emptyset$ if $j\neq k$.  Take $F_k= \operatorname{span}\left\{\varphi_{1}, \varphi_{2}, \ldots, \varphi _k \right\}$, it is clear that $\operatorname{dim} F_k = k$. 

\noindent Denote 
$$
\ell _k = \inf  _{u\in S_{\beta _1} \cap F_k} \int _\Omega \vert  u \vert    ^r \dx >0 \quad \text { and } \quad  \mu  _k = \sup _{u \in S_{\beta _1} \cap F_k}  \int _{\Omega} \vert    \nabla u \vert    ^q \dx    >0,
$$ 
where $S_{\beta _1}=\left\{u \in \wma\:\vert    \:\Vert u \Vert  _{p, \beta _1} =1\right\}$. Write $E_k (t)=t\left(S_{\beta _1} \cap F_k\right)$ for $0<t < 1$. Obviously, $\gamma\left(E_k (t)\right)=k$, for all $t \in ( 0,1)$. We deduce, $E_k (t) \in \Gamma _k$.

Let $k\in \{1, 2, \ldots\}$. Now, we show that there exists $t_0 \in (0,1)$  such that
$$
\sup _{u \in E_k \left(t_0 \right)} \Phi_{\lambda}(u)<0.
$$
Indeed, we have $\lambda > \frac{r^+}{ \alpha ^- \ell _k}\left(\frac{1}{p^-} t_0 ^{p^- - r^+} + \frac{\mu _k}{ q^-} t_0 ^{q^- - r^+} \right)$ for some $t_0 \in (0,1)$. Then,
$$
\begin{aligned}
& \sup _{u \in E_k (t_0)} \Phi_{\lambda}(u) =  \sup _{v \in S_{\beta _1}  \cap F_k} \Phi _\lambda (t_0 v) \\
&= \sup _{v \in S_{\beta _1}  \cap F_k } \int_{\Omega} \frac{t_0 ^{p}}{p}\vert    \nabla v\vert    ^p + \frac{t_0 ^q}{q}\vert    \nabla v \vert    ^q  \dx - \lambda \int_{\Omega} \frac{\alpha t_0 ^r}{r}\vert  v \vert    ^r \dx  \\
& \leqslant  \frac{1}{p^-} t_0 ^{p^-} + \frac{\mu_k}{q^- } t_0 ^{q^-} - \frac{\lambda  \alpha ^- \ell _k}{r^+}  t_0 ^{r^{+}} <0.
\end{aligned}
$$
This completes the proof.

\end{proof}

\noindent \textbf{Proof of Theorem  \ref{52} \eqref{3}.} By Lemma \ref{32} and Theorem \ref{9}, $\Phi_{\lambda}$ admits a sequence of nontrivial weak solutions $\{ u_k \}$ such that for any $k$, we have
$$
u_k \neq 0, \quad \Phi_{\lambda}^{\prime}\left(u_k\right)=0, \quad \Phi_{\lambda}\left(u_k\right) \leqslant 0, \quad \text { and } u_k\rightarrow 0 \text { in } \wma.
$$
\begin{flushright}
$\square$
\end{flushright}

%%%%%%%%%%%%%%%%%%%%%%%%%%%%%%%%%%%%%%%%%%%%%%%%%%%%%%%%%%%%%%%%%%%%%%%%%%%%%%%%%%%%%%%%%%%%%%%%%%%%%%%%%%%%%%%%%%%%%%%%%%%%%%%%%%%%%%%%%%%%%%%%%%%%%%%%%%%%%%%%%%%%%%%%%%%%%%%%%%%%%%%%%%%%%%%%%%%%%%%%%%%%%%%%%%%%%%%%%%%%%%%%%%%%%%%%%%%%%%%%%%%%%%%%%%%%%%%%%%%%%%%%%%%%%%%%%%%%%%%%%%%%%%%%%%%%%%%%%%%%%%%%%%%%%%%%%%%%%%%%%%%%%%%%%%%%%%%%%%%%%%%%%%%%%%%%%%%%%%%%%%%%%%%%%%%%%%%%%%%%%%%%%%%%%%%%%%%%%%%%%%%%%%%%%%%%%%%%%%%%%%%%%%%%%%%%%%%%%%%%%%%%%%%%%%%%%%%%%%%%%%%%%%%%%%%%%%%%%%%%%%%%%%%%%%%%%%%%%%%%%%%%%%%%%%%%%%%%%%%%%%%%%%%%%%%%%%%%%%%%%%%%%%%%%%%%%%%%%%%%%%%%%%%%%%%%%%%%%%%%%%%%%%%%%%%%%%%%%%%%%%%%%%%%%%%%%%%%%%

\subsection{The case  $r^- <  \min \{p^- , q^-\}$  and   $r^+ <\left( \ma ^\ast \right)^-$}  \label{36}

Since  $r^+ <\left( \ma ^\ast \right)^-$, from Proposition \ref{6} it follows that $\wma$ is continuously embedded in $L^{r(\cdot)} (\Omega)$.  Write
\begin{equation} \label{12} 
C_\ast := \sup _{u\in\wma \backslash \{0\}} \frac{\Vert  u\Vert  _{L^{r(\cdot)}(\Omega)} }{\Vert  u\Vert  _{\ma , 1}}<\infty.
\end{equation}

To prove the Theorem \ref{52} \eqref{4},  we  use the  Ekeland's variational principle (see \cite{de198}):

\begin{theorem}(Ekeland Principle-weak form). Let $(X,d)$ be a complete metric space. Let $\Phi: X \rightarrow \mathbb{R}\cup \{\infty\}$ be lower semicontinuous and bounded below. Then given $\varepsilon >0$ there exist $u_\varepsilon \in X$ such that
$$
\Phi (u_\varepsilon) \leqslant \inf _X \Phi + \varepsilon,
$$
and
$$
\Phi (u_\varepsilon) < \Phi (u) + \varepsilon d(u,u_\varepsilon), \quad \text { for all } u\in X \text { with } u\neq u_\varepsilon.
$$
\end{theorem}

We have the following auxiliary

\begin{lemma} \label{30}
\begin{enumerate}[(i)]
\item \label{10} There is $\Lambda >0$ so that for any $\lambda \in\left(0, \Lambda \right)$ there exist $\rho, a>0$ such that $\Phi_{\lambda}(u) \geqslant a$ for any $u \in\wma$ with $\Vert  u\Vert  _{\ma , 1}=\rho$.
\item \label{11}  There exists $\xi \in \wma$ such that $\xi \geqslant 0$, $\xi \neq 0$ and $\Phi_{\lambda}(t \xi)<0$, for $t>0$ small enough.
\end{enumerate}
\end{lemma}

\begin{proof}
 
\eqref{10} Fix $\rho \in (0, \min \{C_\ast ^{-1}, 1\})$, where $C_\ast$ is given by \eqref{12}. Hence, if $u \in \wma$ with $\Vert  u \Vert  _{\ma  , 1} = \rho$, we have $\Vert u \Vert _{L^{r(\cdot)} (\Omega)} \leqslant C_\ast \rho<1$ and
\begin{equation}\label{42}
\begin{split}
&\Phi_{\lambda}(u)  \geqslant \frac{\min \{1, \beta _1 ^- , \beta _2 ^-\}}{\max \{ p^{+} , q^+\}}\left(\int_{\Omega}\vert    \nabla u\vert    ^{\ma} \dx+\int_{\partial \Omega} \vert    u\vert    ^{\ma} \ds\right)-\frac{\lambda \alpha ^+}{r^-} \int_{\Omega}\vert    u\vert    ^{r} \dx \\
& \geqslant  \frac{\min \{1, \beta _1 ^- , \beta _2 ^-\}}{\max \{ p^{+} , q^+\}} \Vert  u\Vert   _{\ma  , 1} ^{\ma ^{+}}-\frac{\lambda \alpha ^+}{r^- } C_{\ast}^{r^- }\Vert  u\Vert  _{\ma , 1}^{r^- } 
=\rho ^{r^- } \left( \frac{\min \{1, \beta _1 ^- , \beta _2 ^-\}}{\max \{ p^{+} , q^+\}}\rho ^{\ma ^{+} - r^- }-\frac{\lambda \alpha ^+  C_{\ast}^{r^-}}{r^-} \right).
\end{split}
\end{equation}
Choose 
$$
\Lambda =\frac{\min \{1, \beta _1 ^- , \beta _2 ^-\} r^-  \rho ^{2\left(\ma ^{+} - r^- \right)}}{\max \{ p^{+} , q^+\}  \alpha ^+  C_{\ast}^{r ^- }}.
$$
Then, if $\lambda \in (0, \Lambda)$:
$$
\Phi_{\lambda}(u) \geqslant \frac{\rho ^{\ma ^+}\min \{1, \beta _1 ^- , \beta _2 ^-\}(1-\rho^{\ma ^+ -r^- })}{\max \{ p^{+} , q^+\} }>0 .
$$
This completes the proof.\\

%%%%%%%%%%%%%%%%%%%%%%%%%%%%%%%%%%%%%%%%%%%%%%%%%%%%%%%%%%%%%%%%%%%%%%%%%%%%%%%%%%%%%%%%%%%%%%%%%%%%%%%%%%%%%%%%%%%%%%%%%%%%%%%%%%%%%%%%%%%%%%%%%%%%%%%%%%%%%%%%%%%%%%%%%%%%%%%%%%%%%%%%%%%%%%%%%%%%%%%%%%%%%%%%%%

\eqref{11} Hence $r^- < \min \{p^- , q^-\}$, there exists $\varepsilon >0$ such that
$$
r^- + \varepsilon <  \min \{p^- , q^-\}.
$$
Since $r\in C(\bar{\Omega})$, there is an open set $U\subset \Omega$ for which 
$$
\vert r - r ^-  \vert <\varepsilon, \quad \text { in }  U.
$$
Thus, 
$$
r< r^- +\varepsilon < \min \{p^- , q^-\}, \quad \text { in } U .
$$

 Take $\xi = 1$. For all $t \in ( 0,1 )$, we obtain
$$
\begin{aligned}
&\Phi_{\lambda}(t \xi) \leqslant  \frac{     \beta _1 ^+ \vert    \partial \Omega\vert }{p^-} t^{p^-} +   \frac{   \beta _2 ^+ \vert    \partial \Omega\vert  }{q^-} t^{q^-}-  \frac{  \lambda   \alpha ^-  \vert    U\vert }{r^+} t^{r^- +\varepsilon} \\
&= \left( \frac{    \beta _1 ^+ \vert    \partial \Omega\vert }{p^- }  t^{p^- -r^- - \varepsilon} + \frac{    \beta _2 ^+ \vert    \partial \Omega\vert  }{q^-} t^{q^- -r^-  -\varepsilon}-  \frac{    \lambda  \alpha ^- \vert   U \vert   }{r^+} \right)  t^{r^- + \varepsilon}.
\end{aligned}
$$
Then, for any $t>0$ small enough, we have
$$
\Phi_{\lambda}(t \xi)<0 .
$$
The proof is complete.

\end{proof}

\noindent \textbf{Proof of Theorem \ref{52} \eqref{4}.} By Lemma \ref{30} \eqref{10}, we have
$$
\inf _{\partial B_{\rho}} \Phi_{\lambda}>0,
$$
where $\partial B_{\rho}=\left\{u \in \wma \: \vert \: \Vert  u\Vert   _{\ma  , 1} =\rho\right\}$ and $B_\rho = \{u\in \wma \: \vert \: \Vert u \Vert _{\ma , 1} <\rho\}$.

Using the estimate \eqref{42}, it follows that
$$
\Phi_{\lambda}(u) \geqslant -\frac{\lambda \alpha ^+ (C_\ast \rho )^{r^-}}{r^-}  \quad \text { for } u \in B_{\rho}.
$$
Hence, by Lemma \ref{30} \eqref{11},
$$
-\infty< \inf _{\overline{B_{\rho}}} \Phi_{\lambda} <0.
$$

Let
$$
0<\varepsilon<\inf _{\partial B_{\rho}} \Phi_{\lambda}-\inf _{\overline{B_{\rho}}} \Phi_{\lambda}.
$$
Then, by applying Ekeland's variational principle to the functional
$$
\Phi_{\lambda}: \overline{B_{\rho}} \rightarrow \mathbb{R},
$$
there exists $u_{\varepsilon} \in \overline{B_{\rho}}$ such that
$$
\begin{aligned}
&\Phi_{\lambda}\left(u_{\varepsilon}\right) \leqslant \inf _{\overline{B_{\rho}}} \Phi_{\lambda}+\varepsilon ,\\
&\Phi_{\lambda}\left(u_{\varepsilon}\right)<\Phi_{\lambda}(u)+\varepsilon\left\Vert  u-u_{\varepsilon}\right\Vert   _{\ma  , 1},  \quad \text { for all } u\in \overline{B_\rho} \text { with } u \neq u_{\varepsilon}.
\end{aligned}
$$
Since $\Phi_{\lambda}\left(u_{\varepsilon}\right)\leqslant \inf _{\overline{B_{\rho}}} \Phi_{\lambda}+\varepsilon<\inf _{\partial B_{\rho}} \Phi_{\lambda}$, we deduce $u_{\varepsilon} \in B_{\rho}$. Then, for $t>0$ small enough and $v \in B_{1}$, we have
$$
\frac{\Phi_{\lambda}\left(u_{\varepsilon}+t v\right)-\Phi_{\lambda}\left(u_{\varepsilon}\right)}{t}+\varepsilon\Vert  v\Vert  _{\ma  , 1} > 0.
$$
 As $t \rightarrow 0^+$, we obtain
$$
\left\langle\Phi_{\lambda}^{\prime}\left(u_{\varepsilon}\right), v\right\rangle+\varepsilon\Vert  v\Vert  _{\ma  , 1} \geqslant 0, \quad \text { for all } v \in B_{1} .
$$
Hence, $\left\Vert  \Phi_{\lambda}^{\prime}\left(u_{\varepsilon}\right)\right\Vert  _{\wma^\ast } \leqslant \varepsilon$. We deduce that there exists a sequence $\{u_k \} \subset B_{\rho}$ such that
\begin{equation}\label{13}
\Phi_{\lambda} (u_k ) \rightarrow \inf _{\overline{B_{\rho}}} \Phi_{\lambda}  \quad \text { and } \quad \Phi_{\lambda}^{\prime}\left(u_{k}\right) \rightarrow 0 \text { in } \wma ^\ast.
\end{equation}
As in the proof of Lemma \ref{32} \eqref{7}, for a subsequence, we obtain $u_k \rightarrow u$ in $\wma$. Thus, by \eqref{13} we have
$$
\Phi_{\lambda}(u)=   \inf _{\overline{B_{\rho}}} \Phi_{\lambda}  <0 \quad \text { and } \quad \Phi_{\lambda}^{\prime}(u)=0 .
$$
So finishes up the proof. 
\begin{flushright}
$\square$
\end{flushright}

%%%%%%%%%%%%%%%%%%%%%%%%%%%%%%%%%%%%%%%%%%%%%%%%%%%%%%%%%%%%%%%%%%%%%%%%%%%%%%%%%%%%%%%%%%%%%%%%%%%%%%%%%%%%%%%%%%%%%%%%%%%%%%%%%%%%%%%%%%%%%%%%%%%%%%%%%%%%%%%%%%%%%%%%%%%%%%%%%%%%%%%%%%%%%%%%%%%%%%%%%%%%%%%%%%%%%%%%%%%%%%%%%%%%%%%%%%%%%%%%%%%%%%%%%%%%%%%%%%%%%%%%%%%%%%%%%%%%%%%%%%%%%%%%%%%%%%%%%%%%%%%%%%%%%%%%%%%%%%%%%%%%%%%%%%%%%%%%%%%%%%%%%%%%%%%%%%%%%%%%%%%%%%%%%%%%%%%%%%%%%%%%%%%%%%%%%%%%%%%%%%%%%%%%%%%%%%%%%%%%%%%%%%%%%%%%%%%%%%%%%%%%%%%%%%%%%%%%%%%%%%%%%%%%%%%%%%%%%%%%%%%%%%%%%%%%%%%%%%%%%%%%%%%%%%%%%%%%%%%%%%%%%%%%%%%%%%%%%%%%%%%%%%%%%%%%%%%%%%%%%%%%%%%%%%%%%%%%%%%%%%%%%%%%%%%%%%%%%%%%%%%%%%%%%%%%%%%%%%

\subsection{The case $\ma^{+}< r^- \leqslant r^+ <(\ma ^\ast) ^- $}  \label{37}

We will prove the Theorem  \ref{52} \eqref{5} using the Mountain Pass Theorem (see  \cite{yan2013}):

\begin{theorem} (Ambrosetti-Rabinowitz). Let $I\in C^1  (X , \mathbb{R})$ be satisfying the (PS) condition  on the real Banach space $X$. Let $u_0, u_1 \in X$, $c_0 \in \mathbb{R}$ and $R > 0$ be such that
\begin{enumerate}[(i)]
\item $\Vert u_1 - u_0 \Vert >R$,
\item $\max\{I(u_0), I(u_1)\} < c_0 \leqslant I(v)$ for all $v$ such that $\Vert v - u_0\Vert = R$.
\end{enumerate}
Then $I$ has a critical point $u$ with $I(u) = c$, $c \geqslant c_0$; the critical value $c$ is defined by
$$
c = \inf _{p\in K} \sup _{t\in [0,1]} I\left(p(t)\right),
$$
where $K$ denotes the set of all continuous maps $p: [0, 1] \rightarrow X$ with $p(0) = u_0$ and $p(1) = u_1$.
\end{theorem}

We need the following 

\begin{lemma} \label{33}
\begin{enumerate}[(i)]
\item \label{14}  There exist $\eta, b>0$ such that $\Phi_{\lambda}(u) \geqslant b$ for any $u \in \wma$ with $\Vert  u\Vert  _{\wma}=\eta$.
\item \label{15} There is $\zeta \in \wma$ such that $\Vert \zeta \Vert _{\wma} > \eta$  and $\Phi_{\lambda}(\zeta)<0$, where $\eta$ is given in  \eqref{14}.
\item \label{16} The  functional $\Phi_{\lambda}$ satisfies the condition (PS).
\end{enumerate}
\end{lemma}

\begin{proof} 

\eqref{14} Assume that $\Vert u \Vert  _{\ma  , 1}  \leqslant \min \{C_\ast ^{-1}, 1\} $, where $C_\ast$ is given  by \eqref{12}. Then
$$
\begin{aligned}
&\Phi_{\lambda}(u) \geqslant  \frac{\min \{1 , \beta _1 ^- , \beta _2 ^-\}}{\max \{p^{+} , q^+ \} }\left(\int_{\Omega}\vert    \nabla u\vert    ^{\ma} \dx + \int_{\partial \Omega} \vert    u\vert    ^{\ma} \ds \right) -\frac{\lambda \alpha ^+ C_\ast ^{r^-}}{r^{-}} \Vert u \Vert _{\ma , 1} ^{r^-} \\
&\geqslant \left(   \frac{\min \{1 , \beta _1 ^- , \beta _2 ^-\}}{\max \{p^{+} , q^+ \} } -\frac{\lambda \alpha ^+ C_\ast ^{r^-}}{r^-} \Vert u \Vert _{\ma , 1} ^{r^- - \ma ^+}   \right)\Vert  u\Vert   _{\ma , 1}^{\ma ^{+}} .
\end{aligned}
$$
Since $\ma ^+ < r^-$, assertion \eqref{14} follows.\\

%%%%%%%%%%%%%%%%%%%%%%%%%%%%%%%%%%%%%%%%%%%%%%%%%%%%%%%%%%%%%%%%%%%%%%%%%%%%%%%%%%%%%%%%%%%%%%%%%%%%%%%%%%%%%%%%%%%%%%%%%%%%%%%%%%%%%%%%%%%%%%%%%%%%%%%%%%%%%%%%%%%%%%%%%%%%%%%%%%%%%%%%%%%%%%%%%%%%%%%%%%%%%%%%%%

\eqref{15} Take $\xi = 1$.  For $t>1$, we have
\begin{align*}
&\Phi_{\lambda}(t \xi) \leqslant \frac{\beta _1 ^+ \vert \partial\Omega \vert }{p^-} t^{p ^+} +  \frac{ \beta _2 ^+ \vert \partial \Omega \vert}{q^-} t^{q ^+} - \frac{ \lambda \alpha ^-\vert \Omega \vert}{r^+} t^{r^-} \\
& \leqslant \left( \frac{\beta _1 ^+ \vert \partial\Omega \vert }{p^-} t^{p ^+ - r^-} +  \frac{ \beta _2 ^+ \vert \partial \Omega \vert}{q^-} t^{q ^+  - r^-} - \frac{ \lambda \alpha ^-\vert \Omega \vert}{r^+} \right) t^{r^-} .
\end{align*}
Since $\max \{p^+ , q^+\}<r^-$, for $t>1$ large enough, there is $\zeta= t \xi$ such that $\Vert \zeta \Vert _{\wma} > \eta$ and $\Phi_{\lambda}(\zeta)<0$. This completes the proof.\\

%%%%%%%%%%%%%%%%%%%%%%%%%%%%%%%%%%%%%%%%%%%%%%%%%%%%%%%%%%%%%%%%%%%%%%%%%%%%%%%%%%%%%%%%%%%%%%%%%%%%%%%%%%%%%%%%%%%%%%%%%%%%%%%%%%%%%%%%%%%%%%%%%%%%%%%%%%%%%%%%%%%%%%%%%%%%%%%%%%%%%%%%%%%%%%%%%%%%%%%%%%%%%%%%%%

\eqref{16} Let $\{u_{k}\} \subset \wma$ be a sequence such that $\sup _k \Phi_{\lambda}\left(u_k\right)<\infty$ and $\Phi_{\lambda}^{\prime}\left(u_k\right) \rightarrow 0$ in $ \wma ^\ast $. First we prove that $\{u_k\}$ is bounded. We argue by contradiction. Let $\varepsilon >0$ be so that $r^-> \ma ^+ + \varepsilon $, and suppose
$$
\left\Vert  u_k\right\Vert  _{\ma  , 1} \rightarrow \infty , \quad \left( \ma ^+ + \varepsilon\right) \Vert u_k \Vert _{\ma , 1} \geqslant -\langle \Phi _\lambda ^\prime  (u_k) , u_k\rangle , \quad  \text { and }  \quad \left\Vert u_k\right\Vert  _{\ma  , 1} >1  \text { for any } k.
$$
Then
$$
\begin{aligned}
& \sup _h \Phi_{\lambda}\left(u_h\right)+\left\Vert  u_k\right\Vert  _{\ma  , 1} \geqslant  \Phi_{\lambda}\left(u_k\right)-\frac{1}{\ma ^+ + \varepsilon}\left\langle\Phi_{\lambda}^{\prime}\left(u_k\right),u_k\right\rangle \\
&= \int_{\Omega} \left( \frac{1}{p} - \frac{1}{\ma ^+  + \varepsilon} \right)\left\vert \nabla u_k\right\vert ^p + \left( \frac{1}{q} - \frac{1}{\ma ^+  + \varepsilon} \right)\left\vert \nabla u_k\right\vert ^q  \dx \\
& + \int_{\partial \Omega} \left( \frac{1}{p} - \frac{1}{\ma ^+  + \varepsilon} \right) \beta _1\left\vert u_k\right\vert ^p +  \left( \frac{1}{q} - \frac{1}{\ma ^+  + \varepsilon} \right) \beta _2\left\vert u_k\right\vert ^q \ds +\lambda\int_{\Omega}\left(\frac{1}{\ma^+  + \varepsilon}-\frac{1}{r}\right) \alpha \left\vert   u_k\right\vert ^r \dx \\
& \geqslant\left(\frac{1}{\max \{ p^{+} , q^+ \}}-\frac{1}{\ma ^+  + \varepsilon}\right)\min \{1 , \beta _1 ^- ,  \beta _2 ^-\} \left\Vert u_k\right\Vert  _{\ma  , 1} ^{\ma ^{-}}.
\end{aligned}
$$
Since $\max \{ p^{+} , q^+ \} < \ma ^+ + \varepsilon$ and $\ma ^- >1$, we have a contradiction.  So, the sequence $\{ u_k \}$ is bounded in $\wma$ and similar arguments as those used in the proof of Lemma \ref{32} \eqref{7} completes the proof.
\end{proof}

\noindent \textbf{Proof of Theorem \ref{52} \eqref{5}.} From Lemma \ref{33} \eqref{14} and \eqref{15}, we deduce
$$
\max \left\{ \Phi_{\lambda}(0), \Phi_{\lambda}(\zeta) \right\} < b \leqslant \inf _{\Vert  v\Vert _{\wma}  =\eta} \Phi_{\lambda}(v).
$$
By Lemma \ref{33} \eqref{16} and the Mountain Pass Theorem, we deduce that  $\Phi_{\lambda}$  has a critical point $u$ with 
$$
\Phi _\lambda (u)=\inf _{\gamma \in K}    \sup _{t \in[0,1]} \Phi_{\lambda}(\gamma(t)) \geqslant b,
$$
where $K=\{\gamma \in C([0,1], \wma)\:\vert \: \gamma(0)=0 \text { and }\gamma(1)=\zeta\}$. This completes the proof.
\begin{flushright}
$\square$
\end{flushright}

%%%%%%%%%%%%%%%%%%%%%%%%%%%%%%%%%%%%%%%%%%%%%%%%%%%%%%%%%%%%%%%%%%%%%%%%%%%%%%%%%%%%%%%%%%%%%%%%%%%%%%%%%%%%%%%%%%%%%%%%%%%%%%%%%%%%%%%%%%%%%%%%%%%%%%%%%%%%%%%%%%%%%%%%%%%%%%%%%%%%%%%%%%%%%%%%%%%%%%%%%%%%%%%%%%%%%%%%%%%%%%%%%%%%%%%%%%%%%%%%%%%%%%%%%%%%%%%%%%%%%%%%%%%%%%%%%%%%%%%%%%%%%%%%%%%%%%%%%%%%%%%%%%%%%%%%%%%%%%%%%%%%%%%%%%%%%%%%%%%%%%%%%%%%%%%%%%%%%%%%%%%%%%%%%%%%%%%%%%%%%%%%%%%%%%%%%%%%%%%%%%%%%%%%%%%%%%%%%%%%%%%%%%%%%%%%%%%%%%%%%%%%%%%%%%%%%%%%%%%%%%%%%%%%%%%%%%%%%%%%%%%%%%%%%%%%%%%%%%%%%%%%%%%%%%%%%%%%%%%%%%%%%%%%%%%%%%%%%%%%%%%%%%%%%%%%%%%%%%%%%%%%%%%%%%%%%%%%%%%%%%%%%%%%%%%%%%%%%%%%%%%%%%%%%%%%%%%%%%

\subsection{The Case $p^+ < q $ and $q>2$}  \label{38}  Let us  recall the following theorem (see \cite{zeidler}).

\begin{theorem} \label{24} (Lagrange multiplier rule)
Let $X$ and $Y$ be real Banach spaces and let $f:D\rightarrow \mathbb{R}$, $G:D\rightarrow Y$ be $C^1$ functions on the open set $D\subset X$. If $u$ is a minimum  point of $f$ on $\{ x \in D \: \vert \: G(x)= 0 \}$, and $G^\prime (u)$ is a surjective operator. Then there exist $y^\ast \in Y^\ast$ such that
$$
f^\prime (u) + y^\ast \circ G^\prime (u) =  0.
$$
\end{theorem}

Observe that $p^+ <q$ implies $\wma = W^{1,q} (\Omega)$. Define
$$
\mathcal{C}_q := \left\{ u\in W^{1,q} (\Omega)\:\vert \:  \int _{\Omega} \alpha \vert u\vert ^{q-2} u \dx =0 \right\}.
$$
We know from \cite[Theorem 6.2.29]{leszek} that 
\begin{equation*} \label{41}
\sigma _1 :=\inf _{u\in \mathcal{C} _q \backslash \{0\}} \frac{ \int _{\Omega} \vert \nabla u\vert ^q \dx}{\int _{\Omega} \alpha \vert u\vert ^q \dx} >0.
\end{equation*}

Inspired by \cite[Section 2.3.2]{badiale} will be to consider the restriction of $\Phi _{\lambda}$ to the Nehari-type manifold defined by
$$
\mathcal{N}_{\lambda} :=\left\{u \in \mathcal{C} _q\backslash\{0\}  \: \vert \: \left\langle \Phi _\lambda ^{\prime}(u), u\right\rangle=0\right\},
$$
where $\lambda >\sigma_1$. In Lemma \ref{34} below we prove that $\mathcal{N} _\lambda$ is nonempty. We recall that the functional $\Phi _\lambda$ has the following expression 
\begin{equation} \label{17}
\Phi _\lambda (u) = \int_{\Omega} \frac{1}{p}\vert    \nabla u\vert    ^{p} +  \frac{1}{q}\vert    \nabla u\vert    ^{q} \dx  -\lambda \int_{\Omega} \frac{\alpha}{q}\vert    u\vert ^q \dx. 
\end{equation}
Then, by definition,
\begin{equation}\label{39}
 \int_{\Omega}\vert    \nabla u\vert    ^p  +  \vert    \nabla u\vert    ^q \dx  = \lambda \int_{\Omega} \alpha \vert u \vert ^q \dx,
\end{equation}
for all $u\in \mathcal{N} _\lambda$.  Furthermore, for all $u\in \mathcal{N} _\lambda$, we have
\begin{equation} \label{18}
\Phi _\lambda (u)=\int_{\Omega}\frac{q-p}{qp} \vert  \nabla u \vert ^p  \dx.
\end{equation}
Consequently, 
$$
m_{\lambda}:=\inf _{u \in \mathcal{N}_{\lambda}} \Phi_{\lambda}(u)\geqslant 0.
$$

We have the following result needed later.

\begin{lemma} \label{34} 
\begin{enumerate}[(i)]
\item  \label{29} $\mathcal{N}_{\lambda} \neq \varnothing$.
\item \label{19} Every minimizing sequence for $\Phi _{\lambda}$ on $\mathcal{N}_{\lambda}$ is bounded in $W^{1, q}(\Omega)$, \textit{i.e.}, if $\{u_k\} \subset \mathcal{N}_\lambda$ and $ \Phi _\lambda (u_k) \rightarrow m_\lambda$, then  $\sup _{k} \Vert u_k\Vert _{W^{1, q}(\Omega)} <\infty$.  
\item  \label{21} $m_{\lambda}>0$.
\item  \label{25} There exists $u_{\Phi} \in \mathcal{N}_{\lambda}$ such that $\Phi_{\lambda}\left(u_{\Phi}\right)=m_{\lambda}$.
\end{enumerate}
\end{lemma}

\begin{proof} 

\eqref{29}  Since $\lambda>\inf _{u\in \mathcal{C} _q \backslash \{0\}}\frac{\int _\Omega \vert \nabla u\vert ^q \dx}{\int _\Omega \alpha \vert u\vert ^q \dx}$, there exists  $u \in \mathcal{C} _q \backslash\{0\}$  such that
$$
 \int_{\Omega}\left\vert    \nabla u \right\vert^{q} \dx <\lambda \int_{\Omega}\alpha \left\vert   u \right\vert    ^q \dx.
$$
Hence, there exist $t_1, t_2 \in  \mathbb{R}$ so that $0<t_1<1<t_2$ and
$$
\begin{aligned}
\int _\Omega t _2 ^{p-q} \vert \nabla u \vert^p \dx <  \lambda \int_{\Omega}\alpha \left\vert   u \right\vert    ^q \dx - \int_{\Omega}\left\vert    \nabla u \right\vert^{q} \dx < \int _\Omega t _1 ^{p-q} \vert \nabla u \vert^p \dx.
\end{aligned}
$$
Thus, we conclude that there exists $t \in (t_1 , t_2)$ for which
$$
\int _{\Omega} t ^{p -q}\vert \nabla u\vert ^p \dx =\lambda \int _{\Omega} \alpha  \vert u\vert ^q \dx -  \int _{\Omega}  \vert \nabla u\vert ^q \dx.
$$
Therefore, $tu\in \mathcal{N}_\lambda$.\\

%%%%%%%%%%%%%%%%%%%%%%%%%%%%%%%%%%%%%%%%%%%%%%%%%%%%%%%%%%%%%%%%%%%%%%%%%%%%%%%%%%%%%%%%%%%%%%%%%%%%%%%%%%%%%%%%%%%%%%%%%%%%%%%%%%%%%%%%%%%%%%%%%%%%%%%%%%%%%%%%%%%%%%%%%%%%%%%%%%%%%%%%%%%%%%%%%%%%%%%%%%%%%%%%%%

\eqref{19} Let $\{u_k\}$ be a minimizing sequence. We argue by contradiction.  Assume that  $\Vert u_k\Vert _{W^{1,q} (\Omega)} \rightarrow \infty$. Then, by \eqref{17}, it follows that $\int_{\Omega}\left\vert   u_k\right\vert ^{q} \dx \rightarrow \infty$. Set $v_k:=\frac{u_k}{\left\Vert u_k\right\Vert _{L^q (\Omega)}}$. Since $\int_{\Omega}\left\vert    \nabla u_k\right\vert    ^{q} \dx<\lambda \alpha ^+ \int_{\Omega}\left\vert   u_k\right\vert    ^{q} \dx$, we deduce  $\int_{\Omega}\left\vert \nabla v_k \right\vert ^{q} \dx<\lambda \alpha ^+$. Thus, $\{ v_k \}$ is bounded in $W^{1, q}(\Omega)$. From the reflexivity of $W^{1,q} (\Omega)$ and Propositions \ref {6}, it follows that there exist $v_{0}\in W^{1, q}(\Omega)$ and a subsequence (still denoted $\{u_k\}$) such that $v_k \rightharpoonup v_{0}$ in $W^{1, q}(\Omega)$ (hence in $W^{1, p (\cdot)}(\Omega)$ as well),  and $v_k \rightarrow v_{0}$ in $L^{q}(\Omega)$. Hence, by Lebesgue's dominated convergence, $v_{0} \in \mathcal{C} _q$.

By \eqref{18},  we obtain
$$
\int_{\Omega}\left\vert \nabla v_k \right\vert ^{p} \dx \rightarrow 0.
$$
Next, since $v_k \rightharpoonup v_{0}$ in $W^{1, p(\cdot)}(\Omega)$, we infer that
$$
\int_{\Omega}\left\vert    \nabla v_{0}\right\vert    ^{p} \dx \leqslant \liminf _{k \rightarrow \infty} \int_{\Omega}\left\vert \nabla v_k \right\vert ^{p} \dx=0
$$
and consequently $v_{0}$ is a constant function. From $v_{0} \in \mathcal{C} _q$, $v_{0}=0$. It follows that $v_k \rightarrow 0$ in $L^{q}(\Omega)$, which contradicts the fact that $\left\Vert  v_k \right\Vert  _{L^{q}(\Omega)}=1$ for all $k$. Consequently, $\{u_k\}$ must be bounded in $W^{1, q}(\Omega)$.\\

%%%%%%%%%%%%%%%%%%%%%%%%%%%%%%%%%%%%%%%%%%%%%%%%%%%%%%%%%%%%%%%%%%%%%%%%%%%%%%%%%%%%%%%%%%%%%%%%%%%%%%%%%%%%%%%%%%%%%%%%%%%%%%%%%%%%%%%%%%%%%%%%%%%%%%%%%%%%%%%%%%%%%%%%%%%%%%%%%%%%%%%%%%%%%%%%%%%%%%%%%%%%%%%%%%

\eqref{21}  Assume by contradiction that $m_{\lambda}=0$. Let $\{u_k\}\subset\mathcal{N}_{\lambda}$ be a minimizing sequence. By  \eqref{19} we know that $\{u_k\} \subset \mathcal{C} _q \backslash \{0\}$ is bounded in $W^{1, q}(\Omega)$. It follows that there exists $u_{0} \in W^{1, q}(\Omega)$ such that (on a subsequence, again denoted $\{u_k\}$) one has $u_k \rightharpoonup u_{0}$ in $W^{1, q}(\Omega)$, and $u_k \rightarrow u_{0}$ in $L^{q}(\Omega)$. Therefore, $u_{0} \in \mathcal{C} _q$ and
$$
\int_{\Omega}\left\vert    \nabla u_{0}\right\vert    ^{p} \dx \leqslant \liminf _{k \rightarrow \infty} \int_{\Omega}\left\vert    \nabla u_k\right\vert ^{p} \dx=0,
$$
because \eqref{18}. Consequently $u_{0}=0$, since $u_0 \in \mathcal{C}_q$.

\noindent Write $v_k:=u_k /\Vert u_k \Vert  _{L^{q}(\Omega)}$. Then, by \eqref{39} and $q>p^+$,
$$
\int _{\Omega} \vert \nabla v_k\vert ^{p} \dx \rightarrow 0.
$$
From $\int_{\Omega}\left\vert \nabla u_k\right\vert    ^{q} \dx<\lambda \alpha ^+ \int_{\Omega}\left\vert   u_k\right\vert    ^{q} \dx$,  we have  $\{v_k\} \subset \mathcal{C} _q$ is bounded in $W^{1, q}(\Omega)$. It follows that, for a subsequence still denoted $\{u_k\}$, there exists $v_{0} \in \mathcal{C} _q$ such that $v_k \rightharpoonup v_{0}$ in $W^{1, q}(\Omega)$ and $v_k \rightarrow v_{0}$ in $L^{q}(\Omega)$. Next,  we see
$$
\int_{\Omega}\left\vert    \nabla v_{0}\right\vert    ^{p} \dx \leqslant \liminf _{k \rightarrow \infty} \int_{\Omega}\left\vert \nabla v_k  \right\vert    ^{p} \dx=0,
$$
and consequently $v_{0}$ is a constant function. In fact, $v_{0}=0$. Thus, $v_k \rightarrow 0$ in $L^{q}(\Omega)$, which contradicts the fact that $\left\Vert  v_k \right\Vert  _{L^{q}(\Omega)}=1$ for all $k$. Consequently, $m_{\lambda}$ is positive, as asserted.\\

%%%%%%%%%%%%%%%%%%%%%%%%%%%%%%%%%%%%%%%%%%%%%%%%%%%%%%%%%%%%%%%%%%%%%%%%%%%%%%%%%%%%%%%%%%%%%%%%%%%%%%%%%%%%%%%%%%%%%%%%%%%%%%%%%%%%%%%%%%%%%%%%%%%%%%%%%%%%%%%%%%%%%%%%%%%%%%%%%%%%%%%%%%%%%%%%%%%%%%%%%%%%%%%%%%

\eqref{25}  Let $\{u_k\} \subset \mathcal{N}_{\lambda}$ be a minimizing sequence. By  \eqref{19}, $\{u_k\}$ is bounded in $W^{1, q}(\Omega)$. Thus, there exists $u _{\Phi} \in \mathcal{C} _q$ such that   (on a subsequence, again denoted $\{u_k\}$) one has $u_k \rightharpoonup u_{\Phi}$ in $W^{1, q}(\Omega)$, and $u_k \rightarrow u_{\Phi}$ in $L^{q}(\Omega)$. We deduce 
\begin{equation}\label{23}
\Phi _{\lambda}\left(u_{\Phi}\right) \leqslant \liminf _{k \rightarrow \infty} \Phi _{\lambda}\left(u_{k}\right)=m_{\lambda} .
\end{equation}
If $u _{\Phi}=0$,  arguing as in the proof of  \eqref{21}, we are led to a contradiction. Consequently $u_\Phi  \in \mathcal{C} _q\backslash\{0\}$. Now, from \eqref{39}, we deduce
$$
\int_{\Omega}\vert    \nabla u_{\Phi} \vert    ^p  +  \vert    \nabla u_{\Phi} \vert    ^q \dx   \leqslant \lambda \int_{\Omega} \alpha \left \vert u_{\Phi}\right \vert ^q \dx.
$$
If we have equality here, then $u _{\Phi} \in \mathcal{N}_{\lambda}$, and everything is done. Assume the contrary, \textit{i.e.},
\begin{equation}\label{22}
\int_{\Omega}\left\vert    \nabla u_{\Phi} \right\vert    ^p  +  \left\vert    \nabla u_{\Phi} \right\vert    ^q \dx   < \lambda \int_{\Omega} \alpha \left\vert u_{\Phi} \right\vert ^q\dx.
\end{equation}
Let $t>0$ be such that $t u_{\Phi} \in \mathcal{N}_{\lambda}$ (see the proof of \eqref{29}). From this, \eqref{22} and our condition $p^+<q$, one can infer that $t \in(0,1)$. Finally,  by \eqref{18} and \eqref{23},  we have
$$
0<m_{\lambda} \leqslant \Phi _{\lambda}\left(t u_{\Phi}\right) \leqslant t^{p^-} \liminf _{k \rightarrow \infty} \Phi _{\lambda}\left(u_{k}\right)=t^{p^-} m_{\lambda}<m_{\lambda},
$$
which is impossible. Hence, relation \eqref{22} cannot be valid, and consequently we must have $u_{\Phi} \in \mathcal{N}_ \lambda$. Therefore, by \eqref{23},  $\Phi _{\lambda}\left(u_{\Phi}\right)=m_{\lambda}$.

\end{proof}

\noindent \textbf{Proof of Theorem \ref{53} \eqref{28}.} \textit{Steep 1.} Let $u_\Phi \in \mathcal{N}_{\lambda} $ be the minimizer found in Lemma \ref{34} \eqref{25}. In fact $u_\Phi$ is a solution of the minimization problem $\min _{u \in W^{1,q} (\Omega) \backslash\{0\}} \Phi _{\lambda}(u)$, under restrictions
\begin{align}
&G_{1}(u):=\int_{\Omega} \vert \nabla u\vert ^p  + \vert \nabla u\vert ^q \dx  -\lambda \int_{\Omega} \alpha \vert u\vert ^{q} \dx=0, \label{26} \\
&G_{2}(u):=\int_{\Omega}\alpha \vert  u\vert  ^{q-2} u \dx = 0. \label{27}
\end{align}
We choose $X=W^{1,q} (\Omega)$, $Y=\mathbb{R}^{2}$, $D=W^{1,q} (\Omega) \backslash\{0\}$, $f=\Phi _{\lambda}$, $G=\left(G_{1}, G_{2}\right)$. Obviously, the dual $Y^\ast$ can be identified with $\mathbb{R}^{2}$. All the conditions from the statement of Theorem \ref{24} are met, including the surjectivity condition on $G^{\prime}(u_\Phi)$, which means that for any pair $\left(\zeta_{1}, \zeta_{2}\right) \in \mathbb{R}^{2}$, there is a $w \in W^{1,q} (\Omega)$ such that $\left\langle G_{1}^{\prime}(u_\Phi ), w\right\rangle=\zeta_{1}$, $\left\langle G_{2}^{\prime}(u_\Phi), w\right\rangle=\zeta_{2}$. Indeed, choosing $w=a u_\Phi+b$ with $a, b \in \mathbb{R}$, we obtain 
$$
\left\{
\begin{aligned}
a  \int_{\Omega}p \vert    \nabla u_\Phi \vert ^p  +  q\vert \nabla u_\Phi \vert ^q \dx   - \lambda a q \int_{\Omega} \alpha \vert  u_\Phi \vert    ^{q} \dx &=\zeta_{1}, \\
b(q-1) \int_{\Omega}\alpha  \vert  u_\Phi \vert ^{q-2} \dx &=\zeta_{2},
\end{aligned}
\right.
$$
which yields
$$
a  \int_{\Omega}(p-q) \vert \nabla u _\Phi \vert ^p  \dx  =\zeta_{1}, \quad b(q-1) \int_{\Omega} \alpha \vert    u_\Phi\vert    ^{q-2} \dx = \zeta_{2} .
$$
Thus, $a$ and $b$ can be uniquely determined, hence $G^{\prime}(u_\Phi)$ is surjective, as asserted. Consequently, Theorem \ref{24} is applicable to our minimization problem. Specifically, there exist some constants $c, d \in \mathbb{R}$ such that
$$
\begin{aligned}
&\left[\int_{\Omega}\left( \vert   \nabla u_\Phi \vert    ^{p-2} + \vert \nabla u_\Phi \vert ^{q-2} \right)  \nabla u_\Phi   \nabla v   \dx   - \lambda \int_{\Omega} \alpha \vert  u_\Phi \vert    ^{q-2} u_\Phi  v \dx\right] \\
&+c\left[\int_{\Omega}\left( p\vert   \nabla u_\Phi \vert    ^{p-2} + q\vert \nabla u_\Phi \vert ^{q-2} \right)  \nabla u_\Phi   \nabla v   \dx   - \lambda q\int_{\Omega} \alpha \vert  u_\Phi \vert    ^{q-2} u_\Phi v \dx\right]
+d(q-1) \int_{\Omega} \alpha  \vert  u_\Phi \vert ^{q-2} v \dx=0, 
\end{aligned}
$$
for all  $v \in W^{1, q}(\Omega)$.  Testing with $v=1$ above, in view of \eqref{27}, yields
$$
d(q-1) \int_{\Omega}\alpha \vert    u_\Phi \vert    ^{q-2} \dx=0,
$$
then $d=0$. Next, testing with $v =u_\Phi$  and using \eqref{26}, we deduce
$$
c \int_{\Omega}(p-q)\vert \nabla u_\Phi \vert ^{p} \dx  =0,
$$
which implies $c=0$. Therefore, for all $v \in W^{1, q}(\Omega)$,
$$
\int_{\Omega}\left( \vert   \nabla u_\Phi \vert    ^{p-2} + \vert \nabla u_\Phi \vert ^{q-2} \right)  \nabla u_\Phi   \nabla v   \dx   - \lambda \int_{\Omega} \alpha \vert  u_\Phi \vert    ^{q-2} u_\Phi v \dx =0,
$$
\textit{i.e.}, $\lambda$ is an eigenvalue of problem \eqref{1}.\\

\textit{Steep 2.} Note that if $\lambda >0$ is an eigenvalue and $u$ is a related eigenfunction  of \eqref{1}, then testing with $v=1$ in \eqref{51}, we deduce $\int _\Omega \vert u \vert ^{q-2} u\dx =0$. Hence, $u\in \mathcal{C}_q \backslash \{0\}$. Next we prove that any  $\lambda \in (0, \sigma _1]$ is not an eigenvalue of \eqref{1}.  We argue by contradiction. Then there exist $\lambda \in (0, \sigma _1]$ and $u\in \mathcal{C}_q \backslash \{0\}$ such that
$$
\int_{\Omega} \vert    \nabla u\vert    ^p + \vert    \nabla u\vert ^q \dx = \lambda \int_{\Omega}\alpha \vert    u\vert    ^q \dx,
$$ 
Hence,
$$
0\leqslant  (\sigma _1 - \lambda)\int _\Omega  \alpha \vert u \vert ^q\dx \leqslant \int _\Omega \vert    \nabla u\vert ^q \dx - \lambda \int_{\Omega}\alpha \vert    u\vert    ^q \dx < \int_{\Omega} \vert    \nabla u\vert    ^p + \vert    \nabla u\vert ^q \dx - \lambda \int_{\Omega}\alpha \vert    u\vert    ^q \dx =0.
$$
This is a contradiction.
\begin{flushright}
$\square$
\end{flushright}

\subsection{The Case $q < p^- $ and $q>2$}  \label{44} We have that $q<p^-$ implies $\wma = W^{1,p(\cdot)} (\Omega)$.  Define
$$
\mathcal{C}:= \left\{ u\in W^{1,p(\cdot)} (\Omega)\:\vert \:  \int _{\Omega} \alpha \vert u\vert ^{q-2} u \dx =0 \right\}.
$$
Note that $\mathcal{C} \subset \mathcal{C}_q$, then 
\begin{equation} \label{45}
\sigma _2 = \inf _{u\in \mathcal{C}  \backslash \{0\}} \frac{ \int _{\Omega} \vert \nabla u\vert ^q \dx}{\int _{\Omega} \alpha \vert u\vert ^q \dx} >0.
\end{equation}

We need the following result.

\begin{theorem}(see \cite[Theorem 1.38]{yan2013}). \label{50}
Let $X$ be a reflexive Banach space and $I : C \subset X \rightarrow \mathbb{R}$ be weakly lower semicontinuous and assume

\begin{enumerate}[(i)]
\item $C$ is a nonempty bounded weakly closed set in $X$ or
\item $C$ is a nonempty weakly closed set in $X$ and $I$ is weakly coercive on $C$.
\end{enumerate}
Then
\begin{enumerate}[(a)]
\item $\inf _{u\in C} I(u) >-\infty$;
\item there is at least one $u_0\in C$ such that $I(u_0) = \inf _{u\in C} I(u)$.
\end{enumerate}
Moreover, if $u_0$ is an interior point of $C$ and $I$ is Gateaux differentiable at $u_0$, then $I^\prime (u_0) = 0$.
\end{theorem}

Recall that a functional $I : C \subset X \rightarrow \mathbb{R}$ is \textit{weakly lower semicontinuous} at $u_0 \in C$ if for every sequence $\{u_k\} \subset C$ for which $u_k \rightharpoonup u_0$ it follows that $I(u_0) \leqslant \liminf _{k\rightarrow \infty} I(u_k)$.

\begin{lemma}\label{46} Assume that $\lambda >\sigma _2$.
\begin{enumerate}[(i)]
\item  \label{47} $\Phi _\lambda :  \mathcal{C} \rightarrow \mathbb{R}$ is weakly coercive, \textit{i.e.},  $\Phi _\lambda (u)\rightarrow \infty$ as $\Vert u\Vert _{W^{1,p(\cdot)} (\Omega)} \rightarrow \infty$ on $\mathcal{C}$.
\item \label{48} The functional $\Phi _\lambda :  \mathcal{C} \rightarrow \mathbb{R} $ has a global minimum point, say $w_\Phi \in \mathcal{C}$, such that $\Phi _\lambda (w_\Phi) <0$.
\end{enumerate}
\end{lemma}
\begin{proof}
\eqref{47} Let $u\in  \mathcal{C}$. By \eqref{45} and H\"older's inequality, we have
\begin{equation} \label{49}
\begin{split}
& \sigma _2 \int_{\Omega} \alpha \vert u\vert ^q \dx  \leqslant \int_{\Omega} \vert \nabla u\vert ^q \dx  \leqslant  C_1 (\Omega  , p , q)  \Vert |\nabla u|^q \Vert _{L^{\frac{p}{q}(\cdot)} (\Omega)}\\
& \leqslant C_1 \max \left\{ \left( \int_{\Omega}  \vert \nabla u\vert ^p \dx \right)^{\frac{q}{p^-}}, \left( \int_{\Omega}  \vert \nabla  u\vert ^p \dx \right)^{\frac{q}{p^+}}\right\}.
\end{split}
\end{equation}
Since $[[u]]:=\Vert \nabla u \Vert _{L^{p(\cdot)} (\Omega)} + \Vert  u \Vert _{L^q (\Omega)} $ is equivalent to $\Vert  u \Vert _{W^{ 1 , p(\cdot)} (\Omega)} $, from \eqref{49},
$$
C_3 \Vert  u \Vert _{W^{ 1 , p(\cdot)} (\Omega)}    \leqslant [[u]] \leqslant C_2 (\alpha , \sigma _2  , \Omega , p , q)\max \left\{ \left( \int_{\Omega}  \vert \nabla u\vert ^p \dx \right)^{\frac{1}{p^-}}, \left( \int_{\Omega}  \vert \nabla  u\vert ^p \dx \right)^{\frac{1}{p^+}}\right\}.
$$
where $C_3$ is a suitable positive constant. Then, if $\Vert  u \Vert _{W^{ 1 , p(\cdot)} (\Omega)}  > C_2 /C_3 $,
$$
\Phi _\lambda (u) \geqslant  \left( \int_{\Omega}  \vert \nabla u\vert ^p \dx \right)^{\frac{q}{p^-}} \left[ \frac{1}{p^+}\left( \int_{\Omega}  \vert \nabla u\vert ^p \dx \right)^{1-\frac{q}{p^-}}  - \frac{C_1}{\sigma _2 } \right] + \int_{\Omega} \frac{1}{q}  \vert \nabla u\vert ^q \dx  .
$$
Hence, the conclusion follows.\\

%%%%%%%%%%%%%%%%%%%%%%%%%%%%%%%%%%%%%%%%%%%%%%%%%%%%%%%%%%%%%%%%%%%%%%%%%%%%%%%%%%%%%%%%%%%%%%%%%%%%%%%%%%%%%%%%%%%%%%%%%%%%%%%%%%%%%%%%%%%%%%%%%%%%%%%%%%%%%%%%%%%%%%%%%%%%%%%%%%%%%%%%%%%%%%%%%%%%%%%%%%%%%%%%%%

\eqref{48} The assumptions of Theorem  \ref{50} are satisfied. Thus, there is a global minimum point of $\Phi _\lambda : \mathcal{C} \rightarrow \mathbb{R}$, say  $w_\Phi \in \mathcal{C}$. On the other hand, since $\lambda >\sigma _2 $, there exists $w\in \mathcal{C} \backslash \{0\}$ such that
$$
\int _\Omega  \vert \nabla w\vert ^q  \dx  - \lambda \int _\Omega  \alpha \vert  w\vert ^q \dx <0. 
$$
Then, for small enough  $t>0$, 
$$
\Phi _\lambda (w_\Phi) \leqslant \Phi _\lambda (t w) \leqslant \frac{t^q}{q}  \left( \int _\Omega \frac{q}{p}t^{p-q}\vert \nabla w \vert ^p  \dx + \int _\Omega \vert \nabla w \vert ^q  \dx  - \lambda\int _\Omega  \alpha \vert  w \vert ^q \dx\right)<0.
$$
In particular, this shows that $w_\Phi \neq 0$.

\end{proof}

\noindent \textbf{Proof of Theorem \ref{53} \eqref{43}.} Following the Step 2 of the proof of Theorem \ref{28} we have  that any  $\lambda \in (0, \sigma _2]$ is not an eigenvalue of \eqref{1}. 

Now, assume that $\lambda >\sigma _2$. Let $w_\Phi$ be given by Lemma \ref{46} \eqref{48}. Thus, $w_\Phi$ is a solution of the minimization problem $\min _{u\in W^{1,p(\cdot)} (\Omega)} \Phi _\lambda (u)$, under the restriction
$$
G(u):= \int _{\Omega} \alpha \vert u \vert ^{q-2} u \dx =0.
$$
From Theorem \ref{24}, with $X=W^{1,p(\cdot)} (\Omega)$, $Y=\mathbb{R}$, $D=W^{1,p(\cdot)} (\Omega)$, for some $a\in \mathbb{R}$ we have
$$
\left[\int _\Omega \left(\vert \nabla w_\Phi \vert ^{p-2} + \vert \nabla w_\Phi \vert ^{q-2} \right) \nabla w_\Phi \nabla v \dx  - \lambda \int _\Omega \alpha \vert w_\Phi\vert ^{q-2} w_\Phi v \dx \right]+ a(q-1)\int _\Omega \alpha \vert w_\Phi\vert ^{q-2} v \dx =0,
$$
for all $v\in W^{1,p(\cdot)} (\Omega)$. Testing with $v=1$ above, we deduce
$$
a(q-1)\int _\Omega \alpha \vert w_\Phi\vert ^{q-2} v \dx,
$$
which yields $a=0$. Therefore, $\lambda$ is an eigenvalue of problem \eqref{1}.

\begin{flushright}
$\square$
\end{flushright}

%%%%%%%%%%%%%%%%%%%%%%%%%%%%%%%%%%%%%%%%%%%%%%%%%%%%%%%%%%%%%%%%%%%%%%%%%%%%%%%%%%%%%%%%%%%%%%%%%%%%%%%%%%%%%%%%%%%%%%%%%%%%%%%%%%%%%%%%%%%%%%%%%%%%%%%%%%%%%%%%%%%%%%%%%%%%%%%%%%%%%%%%%%%%%%%%%%%%%%%%%%%%%%%%%%

\bibliographystyle{plain}
\bibliography{ref}

\begin{flushright}
\today
\end{flushright}
\end{document}